\documentclass{article}
\usepackage{graphicx} 

\title{Quasiconformal Normalization of Random Meromorphic Functions}
\author{Michael Iofin}
\date{August 2025}

\usepackage{amsmath} 
\usepackage{amsthm}
\usepackage{amssymb}
\usepackage{hyperref} 
\usepackage{color}
\usepackage{tikz-cd}

\usepackage{dsfont}

\usepackage{enumerate}

\usepackage{caption}

\usetikzlibrary{patterns}

\newtheorem{theorem}{Theorem}
\newtheorem{lemma}[theorem]{Lemma}
\newtheorem{proposition}[theorem]{Proposition} 

\theoremstyle{definition}
\newtheorem*{definition}{Definition} 

\theoremstyle{remark}
\newtheorem*{remark}{Remark}

\theoremstyle{conjecture}
\newtheorem{conjecture}{Conjecture}

\DeclareMathOperator{\Mod}{Mod}

\begin{document}

\maketitle

\begin{abstract}
    We study the conformal type of surfaces spread over the sphere via random quasiconformal maps. Constructing a random Beltrami coefficient on the complex plane, we obtain a locally quasiconformal homeomorphism with prescribed dilatation that is almost surely surjective and, with high probability, approximately linear. This yields a normalization for random meromorphic functions associated to surfaces spread over the sphere, from which we prove that the surfaces are almost surely parabolic and obtain bounds on the growth order of their Nevanlinna characteristic.
\end{abstract}

\section{Introduction}

The uniformization theorem states that every simply connected Riemann surface is conformally equivalent to the open unit disk $\mathbb{D}$, the plane $\mathbb{C}$, or the Riemann sphere $\overline{\mathbb{C}}$ (see \cite{Conformal Invariants}, Chapter 10). Such surfaces are called \emph{hyperbolic}, \emph{parabolic}, or \emph{elliptic}, respectively. An important classical problem in the study of Riemann surfaces is the type problem, which asks to identify the conformal type of a simply connected Riemann surface.

A \emph{surface spread over the sphere} is a pair $(X, p)$, where $X$ is a topological surface and $p: X \to \overline{\mathbb{C}}$ is a continuous, open, and discrete map from $X$ to the Riemann sphere \cite{eremenko}. According to Stoïlov's theorem \cite{stoilov}, there exists a unique conformal structure on $X$ such that the map $p$ is holomorphic. If $X$ is simply connected and $\phi$ is a uniformizing map from either $\mathbb{C}$, $\overline{\mathbb{C}}$, or $\mathbb{D}$ to $X$, then $f=p \circ \phi$ is a meromorphic function, and $X$ is called the Riemann surface of $f^{-1}$. Surfaces spread over the sphere are a natural way to construct Riemann surfaces and work with the inverses of non-injective meromorphic functions. They were initially used to study inverses of polynomial functions, and work on surfaces spread over the sphere has since been generalized to potentially infinite-sheeted surfaces.

Let $s(r)$ be a nondecreasing nonnegative function defined on $[0, \infty)$. The order $\lambda$ and lower order $\underline{\lambda}$ of $s(r)$, respectively, are defined by the formulas
$$\lambda = \limsup_{r \to \infty} \frac{\log s(r)}{\log r}, \quad \quad \underline{\lambda} = \liminf_{r \to \infty} \frac{\log s(r)}{\log r}.$$
The order of $s(r)$ is the infimum of all constants $k$ such that $s(r)=O(r^k)$. A meromorphic function $f$ associated to a surface spread over the sphere has a Nevanlinna characteristic $T_f(r)$. The Nevanlinna characteristic is of interest because it controls the distribution of the roots of $f(z)=a$ for any given $a \in \overline{\mathbb{C}}$. For example, the Picard-Borel theorem~\cite{nevanlinna} states that the order of $T_f(r)$ is the same as the order of $n_f(r, a)$ for all $a$ with at most 2 exceptions; here $n_f(r, a)$ denotes the number of roots of $f(z)=a$ in the disk $\{|z| \leq r\}$. Nevanlinna theory can be used to prove and vastly generalize the classical Little Picard theorem, which states that a nonconstant meromorphic function can omit at most two points of $\overline{\mathbb{C}}$. We define the order and lower order of $f$ as the order and lower order, respectively, of $T_f(r)$. An important problem in the theory of value distribution of meromorphic functions is to determine the order and lower order of a meromorphic function.

We consider a model for constructing Riemann surfaces. We color the cells of an infinite grid of unit squares black and white in a checkerboard pattern. Associate a copy of the upper hemisphere $\{ z \ | \ \textnormal{Im}(z)>0\}$ of the Riemann sphere to each white square and a copy of the lower hemisphere to each black square. Furthermore, to each vertex of the grid associate a point on the boundary $\mathbb{R} \cup \{ \infty \}$ between the two hemispheres. We require that if $v_1, v_2, v_3, v_4$ are the vertices of a cell in the grid in clockwise order, then the boundary points associated to these vertices are also in cyclic order (clockwise or counterclockwise). Each hemisphere has four associated points on its boundary, which we will call \emph{marked boundary points}. Given any two neighboring cells $c$ and $c'$, we glue the hemispheres associated to $c$ and $c'$ by identifying the boundary arcs between the marked points associated to the two vertices that $c$ and $c'$ share.

Let $\mathcal{R}$ be the surface formed by gluing together all the hemispheres associated to the cells of the grid. Note that $\mathcal{R}$ is a surface spread over the sphere, where the projection map $p: \mathcal{R} \to \overline{\mathbb{C}}$ takes each component hemisphere of $\mathcal{R}$ to the corresponding hemisphere on the Riemann sphere $\overline{\mathbb{C}}$. The projection $p$ makes $\mathcal{R}$ into a Riemann surface homeomorphic to the plane.

The preimage $p^{-1}(\mathbb{R} \cup \{\infty\})$ of the extended real line is called the net of $(\mathcal{R}, p)$. Let $\mathcal{S}$ denote the set of surfaces spread over the sphere whose net is homeomorphic to an infinite square grid; the set $\mathcal{S}$ is precisely the collection of surfaces that can be constructed according to the method above. Vinberg \cite{Vinberg} asked the question of what the conformal type of a surface in $\mathcal{S}$ could be. Clearly such a surface can be of parabolic type, as demonstrated by the Riemann surface of $\wp^{-1}$, where $\wp$ is a Weierstrass elliptic function. A paper by Geyer and Merenkov \cite{hyperbolic net} answers Vinberg's question by constructing a hyperbolic surface with a square grid net.

In this paper, we continue the study of Vinberg's question by constructing a probability measure on a subset of $\mathcal{S}$ and showing that almost all surfaces in it are parabolic. Thus the hyperbolic surface is an edge case. Our main tool is the use of randomly constructed quasiconformal mappings, which normalize the meromorphic functions associated to random surfaces in $\mathcal{S}$. Quasiconformal mappings are a generalization of conformal maps that retain many properties of conformal maps, and they are useful because of their flexibility. They were studied extensively by Teichmüller, for example in \cite{Teichmüller}. The study of random quasiconformal mappings was carried out from an analytic point of view in~\cite{tao} and from a geometric point of view in~\cite{ivrii markovic}. Building on these developments, we establish results that extend the scope of the existing theory. An original feature of our work is the application of our results about random quasiconformal maps to the study of conformal type of surfaces in $\mathcal{S}$. It is of interest to know what other applications our theorems about random quasiconformal mappings may have.

Surfaces in $\mathcal{S}$ may be given a Riemannian manifold structure via the pullback of the spherical metric on $\overline{\mathbb{C}}$. Our model constructs random surfaces with constant positive curvature outside of the countable set of vertices between 4 cells. Surfaces spread over the sphere such as ours are a natural extension of the notion of the Riemann surface of a multi-valued algebraic function. In particular, surfaces spread over the sphere allow for fibers of infinite cardinality. Combinatorially, the random Riemann surfaces in $\mathcal{S}$ have a square lattice structure. Conformal structures arising from random lattices have been studied extensively, for example by Angel and Schramm in \cite{UIPT} and by Gill and Rohde in \cite{Random Maps}. Random lattices have become of interest in the study of statistical mechanics and 2 dimensional quantum gravity, where the discretization is used to convert an integral over an infinite dimensional space into a sum over finitely many variables, as in \cite{quantum gravity}. Previous work has focused on random triangulations, and we extend it to random lattices with quadrilateral cells.

In order to construct random quasiconformal mappings, we let a partition $\mathcal{P} = \{\mathcal{D}_1, \mathcal{D}_2, \mathcal{D}_3, \dots \}$ be a countable collection of Jordan regions $\mathcal{D}_i$ with disjoint interiors whose closures cover $\mathbb{C}$. In this paper, we take all of the Jordan regions $\mathcal{D}_i$ to contain their boundary. In Section~\ref{random differentials}, we will describe a model for choosing a Beltrami coefficient $\mu$ randomly on each region $\mathcal{D}$ of the partition. The quantity $\frac{1+|\mu|}{1-|\mu|}$ is not necessarily uniformly bounded on $\mathbb{C}$, meaning that the measurable Riemann mapping theorem is not applicable; however, we will still be able to construct a locally quasiconformal homeomorphism $w^\mu$ with Beltrami coefficient $\mu$. Even though $w^\mu$ is injective, it is not necessarily surjective. For technical reasons, we restrict our attention to \emph{probabilistically bounded} differentials $\mu$ on partitions $\mathcal{P}$ of \emph{bounded geometry}; these will be defined in Section~\ref{random differentials}. We then obtain our first main theorem.

\begin{theorem}\label{almost sure surjectivity}
    If $\mu$ is a probabilistically bounded Beltrami differential on a partition $\mathcal{P}$ with bounded geometry, then the map $w^\mu$ is surjective onto $\mathbb{C}$ with probability 1.
\end{theorem}

For $p \in \mathbb{C}$, let $B(p, R)$ denote the open disk of radius $R$ centered at $p$. In Section~\ref{random differentials} we also define a condition on $\mathcal{P}$ and $\mu$ known as \emph{periodicity}. Our second main theorem states that if we additionally have periodicity, then $w^\mu$ is close to a linear map with high probability on any disk large enough.

\begin{theorem}\label{approximate linearity}
    Let $\mu$ be periodic and $\mathcal{P}$ be a periodic partition with bounded geometry. Then there exists a linear map $A_\mu$ such that for any $\epsilon>0$, there exists a constant $R_\epsilon$ such that if $R\geq R_\epsilon$ is large enough, then $|A_\mu - w^\mu|<\epsilon R$ on $B(0, R)$ with probability $1-\epsilon$. Here $A_\mu$ depends only on the probability distribution according to which $\mu$ is chosen.
\end{theorem}

The usefulness of Theorems~\ref{almost sure surjectivity} and~\ref{approximate linearity} stems from the natural occurrence of random quasiconformal maps in the modification of random meromorphic functions associated with surfaces in $\mathcal{S}$. In particular, we will consider a probability space $(\mathcal{V}, \eta_\mathcal{V})$, where $\mathcal{V}$ is a subset of $\mathcal{S}$ and $\eta_\mathcal{V}$ is a probability measure on $\mathcal{V}$. The following theorem will then be an easy consequence of Theorem~\ref{almost sure surjectivity}.

\begin{theorem} \label{parabolicity}
    Almost all surfaces in $(\mathcal{V}, \eta_\mathcal{V})$ are parabolic.
\end{theorem}

Theorem~\ref{approximate linearity} provides more specific information about the structure of our random quasiconformal maps and thus lets us extract more information about a meromorphic function $f$ defined by a surface in $\mathcal{V}$. In particular, in addition to learning that $f$ is defined on all of $\mathbb{C}$, we are also able to estimate their order of growth using the following theorem.

\begin{theorem}\label{order of growth estimate}
    The meromorphic function $f$ defined by a surface in $(\mathcal{V}, \eta_\mathcal{V})$ almost always has order at least 2 and lower order at most 2.
\end{theorem}

In Section~\ref{background}, we discuss relevant background related to extremal length and quasiconformal mappings. In Section~\ref{random differentials}, we define partitions of bounded geometry, probabilistically bounded random Beltrami differentials, and periodic partitions and differentials. In Section~\ref{percolation}, we will prove some properties of percolation in partitions of bounded geometry. In Section~\ref{sect: rough quasiconformality}, we will use our percolation results to prove ``rough quasiconformality" of random quasiconformal maps, which will allow us to prove Theorem~\ref{almost sure surjectivity}. In Section~\ref{rescaling}, we will generalize the results of Ivrii and Marković~\cite{ivrii markovic} to prove Theorem~\ref{approximate linearity}. In Section~\ref{random riemann surfaces} we define the probability space $(\mathcal{V}, \eta_\mathcal{V})$. In Section~\ref{quasiconformal deformation} we demonstrate the power of our new methods by using Theorems~\ref{almost sure surjectivity} and~\ref{approximate linearity} to prove Theorems~\ref{parabolicity} and~\ref{order of growth estimate}, vastly generalizing the solution to Vinberg's problem~\cite{Vinberg}. Lastly, in Section~\ref{further discussion} we discuss potential directions of future research.

\section{Background}\label{background}

We review relevant background related to the theory of extremal length and quasiconformal mappings in order to establish notation and conventions. We follow Ahlfors' texts~\cite{Conformal Invariants} and~\cite{Ahlfors}. Let $\Gamma$ be a family of rectifiable curves in $\mathbb{C}$. Define a metric $\rho: \mathbb{C} \to \mathbb{R}$ to be \emph{allowable} if it is Borel measurable, nonnegative, and $\iint \rho^2 \ dx \, dy$ does not equal $0$ or $\infty$. Given $\gamma \in \Gamma$ and an allowable $\rho$, define $\rho(\gamma)$ according to the formula \[\rho(\gamma) = \int_\gamma \rho \ |dz|.\] We then have $$L(\rho) = \inf_{\gamma \in \Gamma} \rho(\gamma).$$

Define $A(\rho)$ according to the formula $$A(\rho) = \iint \rho^2 \ dx \, dy.$$ We have the following definition.

\begin{definition}
    The \emph{extremal length} of $\Gamma$, denoted $\lambda(\Gamma)$, is defined by $$\lambda(\Gamma)=\sup_{\rho} \frac{L(\rho)^2}{A(\rho)},$$ where the supremum is over all allowable functions $\rho$.
\end{definition}

The following proposition makes extremal length relevant to our study of conformal type.

\begin{proposition}[{\cite[Corollary of Theorem 3 in Chapter 1]{Ahlfors}}]\label{invariance}
    If $\phi$ is conformal in a region containing a family $\Gamma$ of rectifiable curves, and $\Gamma'=\phi(\Gamma)$, then $\lambda(\Gamma) = \lambda(\Gamma')$.
\end{proposition}

We will also need the proposition below, known as the comparison principle.

\begin{proposition}[{\cite[Theorem 4-1]{Conformal Invariants}}]\label{comparison}
    Let $\Gamma$ and $\Gamma'$ be two curve families such that every curve $\gamma \in \Gamma$ contains some $\gamma' \in \Gamma'$. In that case, $\lambda(\Gamma) \geq \lambda(\Gamma')$.
\end{proposition}

Lastly, the following proposition, known as the composition law, allows us to further compare the extremal lengths of various curve families.

\begin{proposition}[{\cite[Theorem 4-2]{Conformal Invariants}}]\label{composition}
    Let $\Omega_1, \Omega_2$ be disjoint open sets and $\Gamma_1, \Gamma_2$ be families of curves in $\Omega_1, \Omega_2$, respectively. If $\Gamma$ is a third curve family such that each $\gamma \in \Gamma$ contains a $\gamma_1 \in \Gamma_1$ and a $\gamma_2 \in \Gamma_2$, then $\lambda(\Gamma) \geq \lambda(\Gamma_1)+\lambda(\Gamma_2).$
\end{proposition}

In this paper, we take a quadrilateral to mean a Jordan region $Q$ with four vertices on its boundary dividing the boundary into four arcs, referred to as sides. Furthermore, one pair of opposite sides is marked, so that it is possible to consider the family $\Gamma$ of curves in $Q$ connecting the opposite marked sides. We define the modulus of $Q$ by the formula
$$\Mod Q =  \lambda(\Gamma)^{-1}.$$
An orientation preserving homeomorphism $\phi$ on a region $\Omega$ is called $K$ quasiconformal if
$$\frac{1}{K} \Mod Q \leq \Mod\phi(Q) \leq K \, \Mod Q$$
for all quadrilaterals $Q$ in $\Omega$.

Let $A(r, R)$ denote the annulus with center 0, inner radius $r$, and outer radius $R$. Given a topological annulus $A$, let $\Mod A$ denote the extremal length of the family of curves connecting the boundary components of $A$. It is well known that $\Mod A(r, R) = \frac{1}{2\pi} \log\frac{R}{r}$ (see Example 2 from Chapter 1 of~\cite{Ahlfors}).

Proposition~\ref{comparison} leads to the following criterion for parabolicity.

\begin{proposition}\label{parabolicity criterion}
    Let $X$ be a simply connected open Riemann surface, and let $K$ be a compact subset of $X$ with a nonempty interior. Let $\Gamma$ be the family of curves connecting the boundary of $K$ to the boundary of $X$, meaning that each curve in $\Gamma$ starts on the boundary of $K$ and is eventually in the complement of any compact subset of $X$. If $\lambda(\Gamma) = \infty$, then $X$ is parabolic.
\end{proposition}

\begin{proof}
    Suppose for sake of contradiction that $X$ is hyperbolic. By the uniformization theorem, we may let $\phi$ be a conformal map from $X$ to the unit disk $\mathbb{D}$. We may further ensure that $\phi$ maps an interior point of $K$ to 0. We may thus choose $0<r<1$ such that $B(0, r) \subset \phi(K)$. Let $\Gamma'$ be the family of curves connecting the boundary of $B(0, r)$ to the boundary of the unit disk. By Propositions~\ref{invariance} and~\ref{comparison}, we have $$\lambda(\Gamma') \geq \lambda(\phi(\Gamma)) = \lambda(\Gamma)=\infty,$$ so $\lambda(\Gamma')=\infty$. However, the extremal length of the curve family connecting opposite boundary components of the annulus $A(r, 1)$ is $\frac{1}{2\pi} \ln \frac{1}{r}$, and this quantity is clearly not infinite. We have obtained a contradiction, showing that $X$ is parabolic.
\end{proof}

\section{Random Differentials on Partitions of the Plane}\label{random differentials}

Let $\mathcal{P} = \{ \mathcal{D}_1, \mathcal{D}_2, \mathcal{D}_3, \dots \}$ be a partition of the complex plane $\mathbb{C}$ into a countable number of Jordan regions $\mathcal{D}_i$ such that they cover $\mathbb{C}$ and have disjoint interiors. Here we take a Jordan region to contain its boundary. We make several weak assumptions to control the behavior of the partition $\mathcal{P}$. We define the mesh size $d$ of $\mathcal{P}$ by the formula
$$d=\sup_i \{ \textnormal{diam } \mathcal{D}_i \},$$
where $\textnormal{diam } \mathcal{D}_i$ denotes the diameter of $\mathcal{D}_i$. We say that $\mathcal{P}$ has \emph{bounded density} if for all $R>0$, there exists a constant $k_R$ such that for all $p \in \mathbb{C}$, the disk $B(p, R)$ intersects at most $k_R$ regions of $\mathcal{P}$. We then have the following definition.

\begin{definition}
    A partition $\mathcal{P}$ has \emph{bounded geometry} if it has finite mesh size and has bounded density.
\end{definition}

We discuss a construction of a random Beltrami differential in the plane $\mathbb{C}$. For each region $\mathcal{D}_i$, let $\mu_t^{(i)}(z)$ denote a family of Beltrami differentials on $\mathcal{D}_i$ parametrized by $t \in \mathbb{R}$. Here, the value of $t$  is chosen by some probability measure $\nu_i$ on $\mathbb{R}$, so that $\mu_t^{(i)}(z)$ is a randomly chosen function on $\mathcal{D}_i$. From here forward, we make the assumption that for each $t$ and $i$ there exists a constant $k_t^{(i)}<1$ such that $|\mu_t^{(i)} (z)|\leq k_t^{(i)}$ for all $z \in \mathcal{D}_i$. Thus the quasiconformal dilatation $K=\frac{1+|\mu_t^{(i)} (z)|}{1-|\mu_t^{(i)} (z)|}$ is bounded in each $\mathcal{D}_i$. For each $i$, we require $k_t^{(i)}$ to be measurable as a function of $t$.

Let $\mu$ be the random Beltrami coefficient on $\mathbb{C}$ such that the value of $\mu(z)$ in each $\mathcal{D}_i$ is $\mu_t^{(i)}(z)$, where $t$ is chosen independently and at random by $\nu_i$ in each $\mathcal{D}_i$ (the existence of $\mu$ is guaranteed by the Kolmogorov extension theorem). Let $K_\mu(z) = \frac{1+|\mu(z)|}{1-|\mu(z)|}$ be the quasiconformal dilatation associated with $\mu$. Even though $K_\mu$ is bounded in each $\mathcal{D}_i$, we don't necessarily have an upper bound on $K_\mu$ that holds across all $\mathbb{C}$. Thus the measurable Riemann mapping theorem does not guarantee the existence of a quasiconformal homeomorphism $w^\mu: \mathbb{C} \to \mathbb{C}$ whose Beltrami coefficient is $\mu$. However, Ivrii and Marković~\cite{ivrii markovic} outline a method of constructing $w^\mu$. If we consider the truncated Beltrami differential $\mu_R = \mu \cdot \chi_{B(0, R)}$, then $K_{\mu_R}$ is bounded on all of $\mathbb{C}$ because $B(0, R)$ intersects finitely many regions $\mathcal{D}_i$ by the finite density of $\mathcal{P}$ and $K_\mu$ is bounded in each $\mathcal{D}_i$. Thus the measurable Riemann mapping theorem guarantees the existence of a quasiconformal homeomorphism $w_R: \mathbb{C} \to \overline{\mathbb{C}}$ with Beltrami differential $\mu_R$ that fixes the points $-1, 0, 1$. We may then take a convergent subsequence of the $w_R$ as $R \to \infty$ that converges to a map $w^\mu: \mathbb{C} \to \overline{\mathbb{C}}$ with Beltrami differential $\mu$. We may postcompose $w^\mu$ with a Möbius transformation to ensure that it fixes $0, 1, \infty$ (where $w^\mu$ is said to fix $\infty$ if $w^\mu(\mathbb{C}) \subset \mathbb{C}$). The map $w^\mu$ is injective, but it is not necessarily surjective.

We now define what probabilistically bounded means in Theorem~\ref{parabolicity}.

\begin{definition}
    The random differential $\mu$ is \emph{probabilistically bounded} if for all $\epsilon>0$, there exists a constant $k<1$ such that for all $i$, the probability that $|\mu| \leq k$ everywhere on $\mathcal{D}_i$ is at least $1- \epsilon$.
\end{definition}

Since $K_\mu = \frac{1+|\mu|}{1-|\mu|}$, we obtain an equivalent definition below. Indeed, the constant $K$ in the definition below is connected to the constant $k$ in the definition above by the formula $K=\frac{1+k}{1-k}$.

\begin{definition}
    The random differential $\mu$ is \emph{probabilistically bounded} if for all $\epsilon>0$, there exists a constant $K<\infty$ such that for all $i$, the probability that $K_\mu \leq K$ everywhere on $\mathcal{D}_i$ is at least $1- \epsilon$.
\end{definition}

For our second main theorem, we also need some periodicity assumptions. We say that $\mathcal{P}$ is periodic with period $\Gamma$ if there is a two dimensional lattice of vectors $\Gamma$ such that if $v \in \Gamma$, then translation by $v$ is a symmetry of $\mathcal{P}$. We additionally say that a random Beltrami differential $\mu$ on  $\mathcal{P}$ is periodic with period $\Gamma$ if $\mu$ is distributed in the same way in $\mathcal{D}_i$ and $\mathcal{D}_{i'}$ whenever $\mathcal{D}_{i'}$ is the translation of $\mathcal{D}_i$ by $v \in \Gamma$. We show the following lemma about periodic differentials.

\begin{lemma}\label{periodic=>bounded}
    If $\mu$ is a periodic differential on a periodic partition $\mathcal{P}$ of bounded geometry, then it is probabilistically bounded.
\end{lemma}

\begin{proof}
    We first show that for all $i \in \mathbb{N}$ and $\epsilon>0$, we may choose $k_i<1$ such that $|\mu|<k_i$ on $\mathcal{D}_i$ with probability $1-\epsilon$. Recall that $k_t^{(i)}$ is a measurable function of $t$ such that $\mu_t^{(i)}(z) \leq k_t^{(i)}$ on $\mathcal{D}_i$. For given $k<1$, let $B_k^{(i)}$ denote the set of $t \in \mathbb{R}$ such that $k_t^{(i)} \leq k$. Any $t$ belongs to $B_k^{(i)}$ for $k$ large enough, so continuity of measure implies that
    $$\lim_{k \to 1} \nu_i \left(B_k^{(i)}\right) = 1.$$
    Thus it is possible to choose $k_i<1$ large enough that $\nu_i(B_{k_i}^{(i)}) \geq 1-\epsilon$, meaning that $|\mu| \leq k_i^{(t)}\leq k_i$ on $\mathcal{D}_i$ with probability at least $1-\epsilon$.
    
    Let $u$ and $v$ be two vectors that generate $\Gamma$. The parallelogram spanned by $u$ and $v$ is contained within $B(0, R)$ for some $R$ large enough. The finite density of $\mathcal{P}$ implies that this parallelogram intersects finitely many regions $\mathcal{D}_i$. Fix $\epsilon>0$. For each region $\mathcal{D}_i$ intersecting the parallelogram, we may choose $k_i<1$ large enough that $|\mu| \leq k_i$ on $\mathcal{D}_i$ with probability $1-\epsilon$. If we let $k<1$ be the maximum of the $k_i$ across all $\mathcal{D}_i$ intersecting the parallelogram spanned by $u$ and $v$, it follows by periodicity that $|\mu| \leq k$ with probability $1-\epsilon$ in any region $\mathcal{D}_i$. Thus $\mu$ is probabilistically bounded.
\end{proof}

Suppose now that $\mu$ is a periodic Beltrami coefficient on a partition $\mathcal{P}$ of bounded geometry. As a consequence of Lemma~\ref{periodic=>bounded}, we may apply Theorem~\ref{almost sure surjectivity} to conclude that $w^\mu$ is almost surely surjective. Thus $w^\mu$ is defined by its Beltrami coefficient $\mu$ up to post-composition by a biholomorphism of the plane $\mathbb{C}$, which must be of the form $z \mapsto az+b$. From here forward, we require that $w^\mu$ fixes $0, 1, \infty$; this almost surely determines $w^\mu$ uniquely.

\section{Percolation}\label{percolation}

Let $\mathcal{P}$ be a partition of bounded geometry. We start by proving a lemma about the rate of growth of the number of regions $\mathcal{D}_i$ intersecting the disk $B(0, R)$.

\begin{lemma}\label{quadratic growth}
    Let $\mathcal{P}$ be a partition of bounded geometry. Let $k(R)$ denote the number of regions $\mathcal{D}_i$ that $B(0, R)$ intersects. Then there exist constants $m$ and $M$ depending only on $\mathcal{P}$ such that for all $R$ large enough, we have that
    $$mR^2 \leq k(R) \leq MR^2.$$
\end{lemma}

\begin{proof}
    To prove the upper bound $k(R) \leq MR^2$, note that there exists a constant $M'$ such that $B(0, R)$ can be covered by $M'R^2$ disks of radius 1 for $R$ large enough. By the definition of bounded density, each of those disks intersects at most $k_1$ regions, so $B(0, R)$ intersects at most $k_1M'R^2$ regions. The upper bound is thus proved with $M=k_1M'$.

    Let $d$ denote the mesh size of $\mathcal{P}$, meaning that
    $d \geq \textnormal{diam } \mathcal{D}_i$ 
    for all $i$. To prove the lower bound, note that there exists a constant $m$ such that for all $R$ large enough, the interior of $B(0, R)$ contains at least $mR^2$ squares of side length $2d$ that don't intersect each other, even at the boundary. To each square $Q$ we can associate a region $\mathcal{D}_Q$ that contains the center of $Q$. Note that the regions $\mathcal{D}_Q$ are all distinct, since clearly a single region $\mathcal{D}$ of diameter at most $d$ cannot contain the centers of two squares, which are more than $2d$ apart. It follows that $B(0, R)$ contains at least $mR^2$ regions.
\end{proof}

Note that it follows from the proof of Lemma~\ref{quadratic growth} that the finiteness of $k_1$ implies the finiteness of $k_R$ for all $R$, since $k_R \leq k_1M'R^2$. Thus bounded density can be thought of as a local condition in the sense that the finiteness of $k_R$ for any $R$ implies the finiteness of $k_R$ for all $R$. 

We now describe a percolation process on $\mathcal{P}$. Fix a constant $0 \leq r \leq 1$ and color each region $\mathcal{D}_i$ blue or yellow independently and at random such that the probability that any given region is yellow is at most $r$ (note that the probabilities that two different regions are yellow don't have to be the same, as long as they are both independent and don't exceed $r$). In this case, $r$ is called the \emph{percolation parameter}. For any rectifiable curve $\gamma$, let its chemical length $d_\textnormal{chem}(\gamma)$ be defined as the length of its intersection with the blue regions. For any two points $x$ and $y$, let $d_\textnormal{chem}(x, y)$ denote the infimum of $d_\textnormal{chem}(\gamma)$ across all curves $\gamma$ connecting $x$ and $y$. The Euclidean distance between $x$ and $y$ will be denoted by $d(x, y)$. We are now ready to prove our main lemma about percolation, which states that the chemical distance is comparable to the Euclidean distance with high probability. The lemma is motivated by Lemma 4.2 of~\cite{ivrii markovic}.

\begin{lemma}\label{percolation lemma}
    Let $\mathcal{P}$ be a partition of bounded geometry. Then there exists a positive constant $r < \frac{1}{2}$ dependent on $\mathcal{P}$ only such that if the percolation parameter is at most $r$, then for any $N$ large enough the following holds with probability at least $1 - \frac{1}{N^2}$: for any two points $x, y \in B(0, N)$ with $d(x, y) \geq \log N$ we have $$\frac{1}{10} \cdot  d(x, y) \leq d_\textnormal{chem}(x, y) \leq d(x, y).$$
\end{lemma}

\begin{proof}
    Let $d$ denote the mesh size of $\mathcal{P}$. Fix $R>3d$. Let the $R$-neighborhood of a region $\mathcal{D}_i$ denote the set of all points whose distance to $\mathcal{D}_i$ is less than $R$. Whenever two regions are at a distance at least $R$ from each other, then their $\frac{R}{3}$-neighborhoods do not intersect any regions in common.
    
    Consider two points $x, y \in B(0, N)$. We construct a graph $G=(V, E)$ as follows, where the vertex set $V$ consists of regions $\mathcal{D}$ in $\mathcal{P}$ and the edge set $E$ consists of pairs of regions in $V$. Add the region containing $x$ to $V$. Consider a path $\gamma$ leading from $x$ to $y$. As we travel along $\gamma$, if we enter a region that is at a distance of at least $R$ from any other region in $V$, we add the new region to $V$. Lastly, construct an edge between two elements of $V$ if and only if the distance between the corresponding regions is at most $2R$.
    
    We claim that $G$ is connected. Indeed, we may show that as we travel along $\gamma$, every region we add to $V$ belongs to the same connected component of $G$ as the region containing $x$. If we add region $\mathcal{D}$ to $V$ as we travel along $\gamma$, let $\mathcal{D}'$ denote the region $\gamma$ was in before entering $\mathcal{D}$. Then $\mathcal{D}'$ is not in $V$, so its distance to a region in $V$ is at most $R$. Thus the distance from $\mathcal{D}$ to a region in $V$ is at most $R+d < 2R$. It follows that $\mathcal{D}$ is in the same connected component as the previous regions in $V$. Thus $G$ is connected.
    
    Lastly, note that $V$ must contain a region whose distance to $y$ is at most $2R$. Indeed, if the region containing $y$ is not in $V$, then its distance to some region in $V$ is at most $R$, and the distance from $y$ to a region in $V$ is at most $R+d < 2R$. We may thus choose a sequence $\sigma$ of regions in $V$ starting with the region containing $x$ such that any two consecutive regions in $\sigma$ are at a distance of at most $2R$ apart, and the last region is at a distance of at most $2R$ from $y$. By construction, each region in $\sigma$ is at a distance of at least $R$ from its neighbor. Let $L$ denote the number of regions in $\sigma$.
    
    Call a region $\mathcal{D}$ in $\sigma$ \emph{insular} if each region intersecting the $\frac{R}{3}$-neighborhood of $\mathcal{D}$ is blue. If at least $\frac{9}{10}$ of the regions in $\sigma$ are insular, it follows that
    \begin{equation}\label{chem length}
        d_\textnormal{chem}( \gamma) \geq \frac{9}{10}\frac{R}{3}L.
    \end{equation}
    Choose a sequence of points $x=x_1, x_2, \dots, x_L$ such that $x_i$ belongs to the $i$th region in $\sigma$. Since any two regions in $V$ are at most $2R$ apart and the diameter of each region is at most $d$, it follows that $d(x_i, x_{i+1}) \leq 2R+2d<3R$. Similarly, $d(x_L, y)<3R$. It follows that
    \begin{equation}\label{dist length}
        L \geq \frac{d(x, y)}{3R}
    \end{equation}
    Substituting Equation~\ref{dist length} into Equation~\ref{chem length}, we obtain
    $$d_\textnormal{chem}(\gamma) \geq \frac{9}{10} \frac{R}{3} \frac{d(x, y)}{3R} = \frac{1}{10}d(x, y).$$
    Our strategy is now to show that with high probability, for all possible sequences $\sigma$ at least $\frac{9}{10}$ of the regions in $\sigma$ are insular.

    If $d(x, y) \geq \log N$, then Equation~\ref{dist length} implies that the length $L$ of an associated sequence $\sigma$ of regions is at least $\frac{\log N}{3R}$. Call a sequence $\sigma$ of regions \emph{associable} if the distance between any two consecutive regions in $\sigma$ is at most $2R$ and the distance between any two regions in $\sigma$ is at least $R$. We establish lower bounds on the probability that for any associable sequence of length $L \geq \frac{\log N}{3R}$, at least $\frac{9}{10}$ of its regions are insular.
    
    Set $k=k_{\frac{R}{3}+d}$ so that the $\frac{R}{3}$-neighborhood of any region intersects at most $k$ regions (this is possible because $\mathcal{P}$ has bounded density). Thus the probability that a region is not insular is at most
    $$1 - (1-r)^k \leq 1-(1-rk) = rk.$$
    The probability that more than $\frac{L}{10}$ regions in an associable sequence of length $L$ are not insular is thus at most
    $$\sum_{j = \lceil \frac{L}{10} \rceil} ^L {L \choose j} (rk)^j \leq 2^L \sum_{j = \lceil \frac{L}{10} \rceil} ^{\infty} (rk)^j \leq 2^{L+1} (rk)^\frac{L}{10},$$
    where the last inequality holds as soon as we set $r$ small enough that $rk \leq 
    \frac{1}{2}$. We used the fact that the distance between any two regions of an associable sequence is at least $R$ to conclude that their $\frac{R}{3}$-neighborhoods don't intersect any regions in common, allowing us to conclude that the probabilities that two regions of the sequence are insular are independent.

    By Lemma~\ref{quadratic growth}, there exists a constant $M$ such that for all $N$ large enough, the disk $B(0, N)$ intersects at most $MN^2$ regions of the partition $\mathcal{P}$. Fix a region $\mathcal{D}$ of an associable sequence $\sigma$. Any region whose distance to $\mathcal{D}$ is at most $2R$ must intersect the disk of radius $2R+d<3R$ centered around any point in $\mathcal{D}$, meaning that there are at most $k_{3R}$ choices for the next region in $\sigma$ after $\mathcal{D}$. Thus the number of associable sequences of length $L$ is at most $MN^2 k_{3R}^L$, since there are at most $MN^2$ choices of starting region. Thus the probability that there exists a sequence of length $L$ for which at least $\frac{L}{10}$ regions aren't insular is at most $MN^2 k_{3R}^L \cdot 2^{L+1}(rk)^{\frac{L}{10}} = 2MN^2 \alpha_r^L$, where $\alpha_r = 2k_{3R}(rk)^{\frac{1}{10}}$. Summing over all $L \geq \frac{\log N}{3R}$, we obtain that
    \begin{align*}
        \sum_{L = \lceil \frac{\log N}{3R} \rceil}^\infty 2MN^2\alpha_r^L &= 2MN^2 \sum_{L = \lceil \frac{\log N}{3R} \rceil}^\infty \alpha_r^L \\
        &\leq 4MN^2 \alpha_r^{(\log N)/3R}\\
        &= 4MN^2  N^{ (\log \alpha_r)/3R} \leq \frac{1}{N^2}
    \end{align*}
    for all $N$ large enough if we choose $r$ small enough. Note that $r$ does not depend on $N$ as $N \to \infty$.

    We thus have that if we choose $r$ small enough, then with probability at least $1-\frac{1}{N^2}$ \emph{all} associable sequences of length at least $\frac{\log N}{3R}$ satisfy the property that at least $\frac{9}{10}$ of their regions are insular. Thus for any pair of points $x, y$ with $d(x, y) \geq \log N$, we have that
    $$d_\textnormal{chem} (\gamma) \geq \frac{1}{10} d(x, y)$$
    for any curve $\gamma$ connecting $x$ and $y$. It follows that with probability at least $1-\frac{1}{N^2}$, we have
    $$d_\textnormal{chem}(x, y) \geq \frac{1}{10} d(x, y)$$
    for any pair of points $x, y$ with $d(x, y) \geq \log N$. The proof is complete.

\end{proof}

\section{Rough Quasiconformality}\label{sect: rough quasiconformality}

In this section, we introduce a notion of ``rough quasiconformality" and prove that the mapping $w^\mu$ is roughly quasiconformal with high probability. We then use our results to prove Theorem~\ref{almost sure surjectivity}. The definitions and proofs are largely based on the work of Ivrii and Marković~\cite{ivrii markovic}.

\begin{definition}
    An orientation preserving homeomorphism $\phi$ is called $(K, \epsilon)$ \emph{roughly quasiconformal} if it changes the modulus of any rectangle with side lengths of at least $\epsilon$ by a factor of at most $K$.
\end{definition}

We now show that the random quasiconformal map $w^{\mu}$ is roughly quasiconformal with high probability on $B(0, N)$.

\begin{lemma}\label{rough quasiconformality}
    Let $\mu$ be a probabilistically bounded Beltrami coefficient on a partition $\mathcal{P}$ with bounded geometry. There is a fixed constant $K$ such that for any $N$ large enough, $w^\mu$ is $(K, \log N)$ roughly quasiconformal on $B(0, N)$ with probability $1-\frac{1}{N^2}.$
\end{lemma}

\begin{proof}
    By probabilistic boundedness of $\mu$, it is possible to choose a constant $K>1$ such that the probability that
    $$K_\mu \leq K'$$
    is at least $1-r$, where $r$ is the constant whose existence is guaranteed by Lemma~\ref{percolation lemma}. Color a region $\mathcal{D}$ of the partition yellow if $K_\mu > K'$, and color it blue otherwise. The regions are colored independently, and the probability that a region is yellow is at most $r$. It follows by Lemma~\ref{percolation lemma} that for all $N$ large enough, we have with probability at least $1-\frac{1}{N^2}$ that
    $$d_\textnormal{chem}(x, y) \geq \frac{1}{10} d(x, y)$$
    for all $x, y$ with $d(x, y) \geq \log N$.

    Fix a rectangle $\textbf{R} \subset B(0, N)$ with side lengths $\ell_1, \ell_2 \geq \log N$ such that the sides of length $\ell_1$ are marked. We establish upper bounds on $\Mod w^\mu (\textbf{R})$. To this end, let $\mathcal{B}$ denote the union of the blue regions, and set the metric $\rho_\textbf{R} = \chi_{\mathcal{B} \cap \textbf{R}}$. Note that $A(\rho_\textbf{R}) \leq \ell_1 \ell_2$. Consider also the metric defined by
    $$\rho^*_\textbf{R}(w) = [\textnormal{Jac}(w^\mu)^{-1}(w)]^{\frac{1}{2}} \cdot \chi_{w^\mu(\mathcal{B} \cap \textbf{R})} (w),$$
    so that $A(\rho_\textbf{R}^*) = A(\rho_\textbf{R}) \leq \ell_1\ell_2$.

    Let $\Gamma$ denote the family of curves connecting opposite marked sides of length $\ell_1$ of \textbf{R}. Since $w^\mu$ is $K'$ quasiconformal in the blue regions, we have that
    $$\rho_\textbf{R}^*(w^\mu(\gamma)) \geq \frac{1}{\sqrt{K'}} \rho_\textbf{R}(\gamma)$$
    for all $\gamma \in \Gamma$. Since the distance between the endpoints of $\gamma$ is at least $\ell_2 \geq \log N$, it follows from Lemma~\ref{percolation lemma} that $\rho_\textbf{R}(\gamma) \geq \frac{\ell_2}{10}$ with probability $1-\frac{1}{N^2}$. Denoting by $\Gamma^*$ the family of curves connecting opposite marked sides of $w^\mu(\textbf{R})$, it follows that
    $$L_{\rho_\textbf{R}^*}(\Gamma^*) \geq \frac{1}{10\sqrt{K'}} \ell_2.$$
    Thus
    $$\Mod w^\mu(\textbf{R}) = \lambda(\Gamma^*)^{-1} \leq \frac{A(\rho_\textbf{R}^*)}{L_{\rho_\textbf{R}^*}(\Gamma^*)^2} \leq 100K' \frac{\ell_1 \ell_2}{\ell_2^2} = 100K' \Mod \textbf{R}$$
    for all \textbf{R} with probability $1-\frac{1}{N^2}$. It similarly follows that
    $$\Mod w^\mu(\textbf{R}) \geq \frac{1}{100K'} \ \Mod \textbf{R}$$
    by considering the family of curves connecting the unmarked pair of sides of \textbf{R}. Thus $w^\mu$ is $(100K', \log N)$ roughly quasiconformal with probability $1-\frac{1}{N^2}$, and the lemma is proved with $K=100K'$.
    
\end{proof}

Using Lemma~\ref{rough quasiconformality}, we are ready to begin the proof of Theorem~\ref{almost sure surjectivity}.

\begin{proof}[Proof of Theorem~\ref{almost sure surjectivity}]

Let $Q(n, N)$ denote the topological annulus consisting of the square $[-N, N]^2$ with the square $[-n, n]^2$ removed from its center. We have that $\Mod Q(n, N)$ denotes the extremal length of the family of curves connecting opposite boundary components of $Q(n ,N)$. We establish lower bounds on $\Mod \phi(Q(N, 2N))$, where $\phi$ is any $(K, N)$ roughly quasiconformal map.

Let $R_1$ denote the rectangle $[N, 2N] \times [-2N, 2N]$, and let $R_2, R_3, R_4$ denote the rotations of $R_1$ by $90^{\circ}, 180^{\circ}, 270^{\circ}$ degrees about the origin, respectively. Any path $\gamma$ that connects opposite boundary components of $Q(N, 2N)$ must cross $R_i$ for some $1 \leq i \leq 4$. Indeed, suppose without loss of generality that the ending point of $\gamma$ is on the side $x=2N$. Since the starting point of $\gamma$ has $x$-coordinate at most $N$, it follows that $\gamma$ crosses $R_1$. The rectangles $R_i$ are depicted in Figure~\ref{fig:Q(N, 2N)}.

Therefore a path $\gamma$ connecting opposite boundary components of $\phi(Q(N, 2N))$ must cross $\phi(R_i)$ for some $1 \leq i \leq 4$. Since the modulus of $R_i$ is $\frac{1}{4}$, it follows that the modulus of $\phi(R_i)$ is at least $\frac{1}{4K}$. For any $\epsilon>0$, there exists a metric $\rho_i$ such that $A(\rho_i) = 1$ and the $\rho_i (\gamma)^2 \geq \frac{1}{4K}-\epsilon$ for any curve $\gamma$ crossing $\phi(R_i)$.

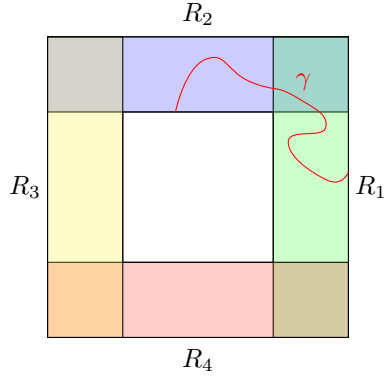
\begin{figure}[h]
\centering

\begin{tikzpicture}[scale=0.5]
    \def\N{2}
    \draw (-\N, \N) -- (\N, \N) -- (\N, -\N) -- (-\N, -\N) -- cycle;
    \draw (-2*\N, 2*\N) -- (2*\N, 2*\N) -- (2*\N, -2*\N) -- (-2*\N, -2*\N) -- cycle;

    \filldraw[fill=blue, draw=black, fill opacity=0.2] (-2*\N, 2*\N) rectangle (2*\N, \N);

    \filldraw[fill=green, draw=black, fill opacity=0.2] (\N, 2*\N) rectangle (2*\N, -2*\N);

    \filldraw[fill=red, draw=black, fill opacity=0.2] (2*\N, -2*\N) rectangle (-2*\N, -\N);

    \filldraw[fill=yellow, draw=black, fill opacity=0.2] (-\N, -2*\N) rectangle (-2*\N, 2*\N);

    \draw [red] plot [smooth, tension=1] coordinates{
        (-\N*0.3, \N)
        (\N*0.1, \N*1.7)
        (\N*0.7, \N*1.4)
        (\N*1.3, \N*1.2)
        (\N*1.7, \N*0.8)
        (\N*1.2, \N*0.6)
        (\N*1.7, \N*0.1)
        (2*\N, \N*0.2)
    };

    \node [red] at (\N*1.4, \N*1.4) {$\gamma$};

    \node at (\N*2.3, 0) {$R_1$};
    \node at (\N*0, \N*2.3) {$R_2$};
    \node at (-\N*2.3, 0) {$R_3$};
    \node at (0, -\N*2.3) {$R_4$};

\end{tikzpicture}
    
\caption{Curve $\gamma$ connecting opposite boundary components of $Q(N, 2N)$}
\label{fig:Q(N, 2N)}
\end{figure}

Consider now the metric $\rho=\rho_1+\dots + \rho_4$. By Muirhead's inequality, $\rho^2 \leq 4(\rho_1^2+\rho_2^2+\rho_3^2+\rho_4^2)$. Thus $A(\rho) \leq  16$ and $\rho(\gamma)^2\geq \frac{1}{4K}-\epsilon$ for any curve $\gamma$ connecting opposite boundary components of $\phi(Q(N, 2N))$. It follows that the extremal length of the curve family crossing $\phi(Q(N, 2N))$ is at least
$$\frac{L(\rho)^2}{A(\rho)} \geq \frac{1}{64K} - \frac{\epsilon}{16}.$$
Letting $\epsilon$ tend to $0$, we obtain
$$\Mod \phi(Q(N, 2N)) \geq \frac{1}{64K}$$

Note that $Q(2^kN, 2^{k+1}N)$ is contained in $B(0, 2^{k+1}\sqrt{2}N)$. Choose $N$ large enough that $N \geq \log (2\sqrt{2}N)$ and Lemma~\ref{rough quasiconformality} applies to $B(0, N')$ for all $N' \geq 2\sqrt{2}N$. Consider the sequence of nested annuli
$$Q(N, 2N), Q(2N, 4N), Q(4N, 8N), \dots$$
With probability $1-\frac{1}{2^{2k+3}N^2}$, we have that $w^\mu$ is $(K, 2^kN)$ roughly quasiconformal in $B(0, 2^{k+1}\sqrt{2}N)$; in fact, $w^\mu$ satisfies the stronger assumption of $(K, \log(2^{k+1}\sqrt{2}N))$ rough quasiconformality. We therefore have that the extremal length of the curve family connecting opposite boundary components of $Q(2^kN, 2^{k+1}N)$ is at least $\frac{1}{64K}$. By the Borel-Cantelli lemma, this holds almost surely for all but finitely many $k$. By Proposition~\ref{composition}, the extremal length of the curve family connecting $w^\mu([-N, N]^2)$ to the boundary of $w^\mu(\mathbb{C})$ is at least $\frac{1}{64K}+\frac{1}{64K}+\dots  = \infty$. By Proposition~\ref{parabolicity criterion}, it follows that $w^\mu(\mathbb{C})$ is parabolic almost surely. This means that $w^\mu$ is surjective almost surely, completing the proof.

\end{proof}

\section{Rescaling}\label{rescaling}

We now turn to proving Theorem~\ref{approximate linearity}. Given $\delta$, let $I_\delta$ denote the map $z \mapsto \delta z$. Let $\delta \mathcal{P}$ denote the partition $\mathcal{P}$ scaled by applying the map $I_\delta$, meaning that the mesh size of $\delta \mathcal{P}$ is smaller than the mesh size of $\mathcal{P}$ by a factor of $\delta$. We may consider the corresponding random Beltrami differential $\mu_\delta = \mu \circ I_\delta^{-1}$ on the partition $\delta \mathcal{P}$. Lastly, let $w_\delta^\mu$ denote the quasiconformal map with Beltrami differential $\mu_\delta$. We rephrase Lemma~\ref{rough quasiconformality} for $w_\delta^\mu$.

\begin{lemma}\label{rescaled rough quasiconformality}
    Fix any $N, \epsilon>0$. For any $\delta$ small enough, the map $w^\mu_\delta$ is $(K, \epsilon)$ roughly quasiconformal on $B(0, N)$ with probability at least $1-\epsilon$.
\end{lemma}

\begin{proof}
    Showing that $w^\mu_\delta$ is $(K, \epsilon)$ roughly quasiconformal on $B(0, N)$ is equivalent to showing that $w^\mu$ is $(K, \delta^{-1}\epsilon )$ roughly quasiconformal on $B(0, \delta^{-1}N)$. However, this is true with probability at least $1- \frac{\delta^2}{N^2}$ for $\delta$ small enough that $\delta^{-1}\epsilon  \geq \log (\delta^{-1}N)$ by Lemma~\ref{rough quasiconformality}. It remains to further choose $\delta$ small enough that $\frac{\delta^2}{N^2}<\epsilon$.
\end{proof}

We now obtain a slightly modified form of Theorem~\ref{approximate linearity}.

\begin{theorem}\label{rescaled approximate linearity}
    Let $\mu$ be periodic and $\mathcal{P}$ be a periodic partition with bounded geometry. Then there exists a linear map $A_\mu$ such that for any $\epsilon>0$, there exists a constant $\delta_\epsilon$ such that if $\delta \leq \delta_\epsilon$ is small enough, then $|A_\mu - w_\delta^\mu|<\epsilon $ on $B(0, 1)$ with probability at least $1-\epsilon$. Here $A_\mu$ depends only on the probability distribution according to which $\mu$ is chosen.
\end{theorem}

The proof of the theorem follows verbatim from sections 6 and 7 of~\cite{ivrii markovic} once rough quasiconformality is established with high probability in Lemma~\ref{rescaled rough quasiconformality}. Note that the assumption of periodicity is used in the proof of Lemma 6.3 in~\cite{ivrii markovic}. We now deduce Theorem~\ref{approximate linearity} from Theorem~\ref{rescaled approximate linearity}.

\begin{proof}[Proof of Theorem~\ref{approximate linearity}]
    We continue with the same notation as the statement of Theorem~\ref{approximate linearity}. Set $\delta=\frac{1}{R}$. Consider the map $w=I_\delta \circ w^\mu \circ I_R$. We have that the Beltrami coefficient of $w$ is the coefficient $\mu \circ I_R$ on the partition $\delta \mathcal{P}$, meaning that $w$ is distributed in the same way as $w_\delta^\mu$.
    
    It follows from Theorem~\ref{rescaled approximate linearity} that if $\delta \leq \delta_\epsilon$ is small enough (or equivalently $R\geq \delta_\epsilon^{-1}$ is large enough), then with probability $1-\epsilon$, we have $|w-A_\mu|<\epsilon$ on $B(0, 1)$. Therefore $|I_R \circ w \circ I_\delta - I_R \circ A_\mu \circ I_\delta|<\epsilon R$ on $B(0, R)$. However, since $I_R \circ w \circ I_\delta = w^\mu$ and $I_R \circ A_\mu \circ I_\delta = A_\mu$, it follows that with probability $1-\epsilon$, $$|w^\mu - A_\mu| < \epsilon R$$ on $B(0, R)$ as needed. Thus Theorem~\ref{approximate linearity} is proved with $R_\epsilon=\delta_\epsilon^{-1}$.
    
\end{proof}

\begin{remark}
    Note that Theorem~\ref{approximate linearity} holds when $B(0, R)$ is replaced by $I_R(K)$, where $K$ is an arbitrary compact set and $I_R(K)$ is the dilation of $K$ by a factor of $R$. Indeed, this holds by observing that $I_R(K) \subset B(0, cR)$, where $c$ is a constant so that $K \subset B(0, c)$. It then follows by Theorem~\ref{approximate linearity} that $|w^\mu-A_\mu| < \epsilon'cR$ on $I_R(K)$ for $R$ large enough, and it remains to choose $\epsilon' = c^{-1}\epsilon$.
\end{remark}

\section{Random Riemann Surfaces}\label{random riemann surfaces}

In this section, we define a subset $\mathcal{V}$ of $\mathcal{S}$ and define a probability measure $\eta_\mathcal{V}$ on $\mathcal{V}$. The need to consider a simpler subset of $\mathcal{S}$ arises from the difficulty of parametrizing all of $\mathcal{S}$. As we will see, the generality of Theorems~\ref{almost sure surjectivity} and~\ref{approximate linearity} gives us flexibility in the choice of $\mathcal{V}$, and $\mathcal{V}$ encompasses a wide variety of the surfaces in $\mathcal{S}$. For the remainder of this section, note that a surface in $\mathcal{S}$ is determined by its marked boundary points, so that a probability distribution on $\mathcal{S}$ can be thought of as a probability distribution according to which the marked boundary points are chosen.

View the Riemann sphere $\overline{\mathbb{C}}$ as the subset $\{(x, y, z) \ | \ x^2+y^2+z^2=1 \}$ of $\mathbb{R}^3$, so that the set of points with $z>0$ corresponds to the upper hemisphere $\mathbb{H}^+$ of complex numbers with positive imaginary part, and the set of points with $z<0$ corresponds to the lower hemisphere $\mathbb{H}^-$ of complex numbers with negative imaginary part. Consider the polar coordinate system $(r, \theta)$ on the unit disk $\{(x, y, 0) \ | \ x^2 + y^2 \leq 1 \}$ such that the points $1, \infty, -1, 0$ have coordinates $(1, 0)$, $(1, \frac{\pi}{2})$, $(1, \pi)$, $(1, \frac{3\pi}{2})$, respectively. In the hemispheres $\mathbb{H}^+$ and $\mathbb{H}^-$, define the polar coordinate systems $(r, \theta)$ so that the polar coordinates of a point $(x, y, z)$ in one of the hemispheres are the same as the polar coordinates of $(x, y, 0)$. From now on, call these the hemispherical coordinates on $\mathbb{H}^+$ and $\mathbb{H}^-$.

Consider an infinite rectangular lattice in the plane. Label the vertex at position $(i, j)$ with an integer between $1$ and $4$ depending on the parity of $i$ and $j$ such that the vertices of any cell are labeled $1, 2, 3, 4$ in cyclic order. Choose four points with hemispherical coordinates $(1, \theta_1),  (1, \theta_2), (1, \theta_3), (1, \theta_4)$ and open arc intervals $C_i$, $1 \leq i \leq 4$, so that the $C_i$ are disjoint, their closures cover $\mathbb{R} \cup \{ \infty \}$, and the endpoints of $C_i$ are $(1, \theta_i)$ and $(1, \theta_{i+1})$ (here $\theta_5 = \theta_1$). This is shown in Figure~\ref{fig:vertices}. Lastly, define probability distributions $\eta_i$ on the $C_i$. The distributions $\eta_i$ may for instance be the uniform distributions induced by arc length.

\begin{figure}[h!]
\centering

\begin{tikzpicture}[scale=0.5]
\def\gridsize{5}
\def\sidelength{1.5}
\def\overhang{0.75}

\begin{scope}[shift={(-8, 0)}]
    
    \foreach \i in {0,...,\gridsize}{
        \draw (-\overhang, \i*\sidelength) -- (\gridsize*\sidelength+\overhang, \i*\sidelength);

        \draw (\i*\sidelength, -\overhang) -- (\i*\sidelength, \gridsize*\sidelength+\overhang);

        \foreach \j in {0,...,\gridsize}{
            \node (X) at (\i*\sidelength, \j*\sidelength) {};
            \fill(X) circle (2pt);

            \pgfmathtruncatemacro\testi{mod(\i, 2)}          \pgfmathtruncatemacro\testj{mod(\j, 2)}

            \pgfmathtruncatemacro\test{2*\testi*\testj+3*\testi*(1-\testj)+1*(1-\testi)*\testj+4*(1-\testi)*(1-\testj)}

            \node[above right] at (X) [font=\scriptsize] {\test};

        }
    }
\end{scope}

\def\radius{3}

\begin{scope}[shift={(8, \gridsize/2*\sidelength)}]

    \draw (0, 0) circle (\radius);

    \draw[draw=black, fill=white] (10:\radius) circle (3pt);

    \node[right] at (10:\radius) {$(1, \theta_1)$};

    \draw[draw=black, fill=white] (80:\radius) circle (3pt);

    \node[above] at (80:\radius) {$(1, \theta_2)$};

    \draw[draw=black, fill=white] (180:\radius) circle (3pt);

    \node[left] at (180:\radius) {$(1, \theta_3)$};

    \draw[draw=black, fill=white] (300:\radius) circle (3pt);

    \node[below] at (300:\radius*1.05) {$(1, \theta_4)$};

    \node[above right] at (45:\radius) {$C_1$};

    \node[above left] at (130:\radius) {$C_2$};

    \node[below left] at (240:\radius) {$C_3$};

    \node[below right] at (335:\radius) {$C_4$};

\end{scope}

\end{tikzpicture}

\caption{Labeling the vertices and dividing $\mathbb{R} \cup \{ \infty \}$ into intervals.}
\label{fig:vertices}
\end{figure}
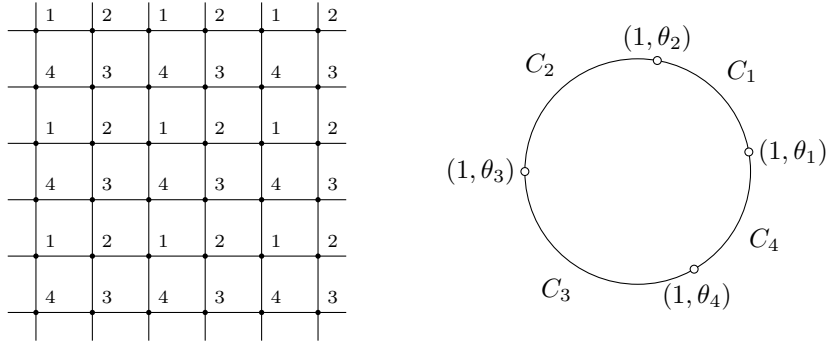

To construct a Riemann surface $\mathcal{R}$ in $\mathcal{S}$, we assign to each vertex labeled $i$ a random marked boundary point chosen by $\eta_i$. Gluing together the surface with these marked points, we obtain a Riemann surface. We let $\mathcal{V}$ be the set of surfaces obtainable in this way. The vertices of a surface in $\mathcal{V}$ are restricted to vary along an interval $C_i$ in order to guarantee that the marked boundary points of each cell are in cyclic order. Now, let $\eta_\mathcal{V}$ be the distribution on $\mathcal{V}$ corresponding to choosing each vertex according to $\eta_i$ (the existence of $\eta_\mathcal{V}$ is guaranteed by the Kolmogorov extension theorem). We thus have a probability space $(\mathcal{V}, \eta_\mathcal{V})$. Note that the definition of $(\mathcal{V}, \eta_\mathcal{V})$ depends on the choice of $\theta_i$.

The set $\mathcal{V}$ is large enough to contain a hyperbolic surface, as shown by Geyer and Merenkov~\cite{hyperbolic net}. However, Theorem~\ref{parabolicity}, which will be proved in Section~\ref{quasiconformal deformation}, states that the set of hyperbolic surfaces has measure 0 in $(\mathcal{V}, \eta_\mathcal{V})$. Thus there almost surely exists a biholomorphism $\phi: \mathcal{R} \to \mathbb{C}$. The surface $\mathcal{R}$ comes equipped with a mapping $p: \mathcal{R} \to \overline{\mathbb{C}}$ which sends each cell in $\mathcal{R}$ to the corresponding hemisphere of $\overline{\mathbb{C}}$. Therefore, a surface $\mathcal{R}$ in $\mathcal{V}$ can be viewed as the Riemann surface of the inverse to the meromorphic function $f=p \circ \phi^{-1}$, which almost always has the plane $\mathbb{C}$ as its domain of definition. Thus $\mathcal{R}$ almost always defines a meromorphic function $f$ in the plane up to precomposition by an automorphism of $\mathbb{C}$, which is of the form $z \mapsto az+b$. If we consider the Nevanlinna characteristic $T_f(r)$ of $f$, defined for $r $ in $[0, \infty)$, we see that precomposing $f$ by an automorphism of $\mathbb{C}$ at most changes the constant in front of $r$ and changes $T_f(r)$ by a bounded additive factor, thereby not changing the order or lower order of $f$. Theorem~\ref{order of growth estimate}, which references \emph{the} order of growth of $f$, therefore makes sense.

\section{Quasiconformal Deformation of an Elliptic Function}\label{quasiconformal deformation}

In this section we prove Theorems~\ref{parabolicity} and~\ref{order of growth estimate} using Theorems~\ref{almost sure surjectivity} and~\ref{approximate linearity}, respectively. We accomplish this by obtaining $f$ as a locally quasiconformal deformation of an elliptic function of order 2.

Let $(1, m_i)$ be the midpoint of $C_i$. Consider the upper hemisphere $\mathbb{H}^+$ with marked boundary points $(1, m_i)$. There exists some rectangle $Q \subset \mathbb{C}$ and a conformal map $\wp: Q \to \overline{\mathbb{C}}$ that takes $Q$ to $\mathbb{H}^+$ and the vertices of $Q$ to the points $(1, m_i)$. Extending $\wp$ by reflection we obtain a doubly periodic function $\wp: \mathbb{C} \to \overline{\mathbb{C}}$ that maps each cell of a rectangular grid in $\mathbb{C}$ onto either the upper hemisphere $\mathbb{H}^+$ or the lower hemisphere $\mathbb{H}^-$.

Define $\mathcal{R}_0$ to be the surface in $\mathcal{V}$ with the marked boundary point $(1, m_i)$ assigned to vertices labeled $i$ (here we are using the vertex labeling from Section~\ref{random riemann surfaces}). Then $\wp$ induces a mapping $\wp_0: \mathbb{C} \to \mathcal{R}_0$ such that $\wp = p_0 \circ \wp_0$, where $p_0: \mathcal{R}_0 \to \overline{\mathbb{C}}$ is the natural projection taking each cell of $\mathcal{R}_0$ to the corresponding hemisphere of $\overline{\mathbb{C}}$.

Given a Riemann surface $\mathcal{R}$ in $\mathcal{V}$ and the associated projection $p: \mathcal{R} \to \overline{\mathbb{C}}$, we deform $\wp_0$ into a function $\tilde \wp_0: \mathbb{C} \to \mathcal{R}$ so that $f=p \circ \tilde \wp_0$, where $f$ is such that $\mathcal{R}$ is the Riemann surface of $f^{-1}$.

Divide each hemisphere of $\mathcal{R}_0$ into eight sectors along the lines $\theta = \theta_i$ and $\theta=m_i$. Consider the preimages of these lines by $\wp_0$. They divide each cell of the rectangular grid in $\mathbb{C}$ into eight triangular Jordan regions, where each triangular region corresponds to a sector of a hemisphere on $\mathcal{R}_0$. This is demonstrated in Figure~\ref{fig:triangles}.

\begin{figure}[h]
    \centering

\begin{tikzpicture}[scale=0.5]
    \def\gridsize{4}
    \def\sidelength{2}
    \def\overhang{1}
    \foreach \i in {0,...,\gridsize}{
        \draw (-\overhang, \i*\sidelength) -- (\gridsize*\sidelength+\overhang, \i*\sidelength);

        \draw (\i*\sidelength, -\overhang) -- (\i*\sidelength, \gridsize*\sidelength+\overhang);

        \foreach \j in {0,...,\gridsize}{
            \begin{scope}[shift={(\i*\sidelength,\j*\sidelength)}]

            \draw[densely dotted] plot [smooth] coordinates{
            (-\sidelength/2, -\sidelength/2)
            (-\sidelength/6, -\sidelength/3)
            (0, 0)
            (\sidelength*0.2, \sidelength*0.4)
            (\sidelength/2, \sidelength/2)};

            \draw[densely dotted] plot [smooth] coordinates{
            (-\sidelength/2, \sidelength/2)
            (-\sidelength/3, \sidelength*0.37)
            (-\sidelength/4, \sidelength*0.15)
            (0, 0)
            (\sidelength/4, -\sidelength*0.2)
            (\sidelength/2, -\sidelength/2)
            };

            \ifnum\i<\gridsize
                \draw[densely dotted] plot [smooth] coordinates{
                (\sidelength/2, \sidelength/2)
                (\sidelength*0.45, \sidelength/4)
                (\sidelength/2, 0)
                (\sidelength*0.55, -\sidelength/3)
                (\sidelength/2, -\sidelength/2)
                };
            \fi

            \ifnum\j>0
                \draw[densely dotted] plot [smooth] coordinates{
                (-\sidelength/2, -\sidelength/2)
                (-\sidelength/4, -\sidelength*0.55)
                (0, -\sidelength/2)
                (\sidelength/4, -\sidelength*0.45)
                (\sidelength/2, -\sidelength/2)
                };
            \fi
            
            \end{scope}

            \node (X) at (\i*\sidelength, \j*\sidelength) {};
            \fill(X) circle (2pt);

            \pgfmathtruncatemacro\testi{mod(\i, 2)}          \pgfmathtruncatemacro\testj{mod(\j, 2)}

            \pgfmathtruncatemacro\test{2*\testi*\testj+3*\testi*(1-\testj)+1*(1-\testi)*\testj+4*(1-\testi)*(1-\testj)}

            \node[above right] at (X) [font=\scriptsize] {\test};

        }
    }

\end{tikzpicture}

    \caption{The cells of the rectangular grid divided into triangular Jordan regions.}
    \label{fig:triangles}
\end{figure}
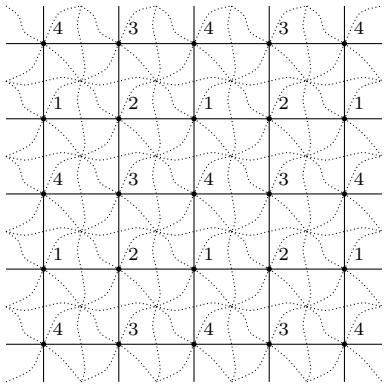

Fix some cell in $\mathcal{R}$. Suppose that the vertex with label $i$ of that cell has associated boundary point with coordinates $(1, \alpha_i)$ for $1 \leq i\leq 4$. Note that $(1, \alpha_i)$ belongs to the open arc interval $C_i$ with endpoints $(1, \theta_i), (1, \theta_{i+1})$. We apply a map of the form $(r, \theta) \mapsto \left(r, k(\theta)\right)$, where $k(\theta)$ is a piecewise linear map such that $(1, m_i)$ is taken to $(1, \alpha_i)$ and $(1, \theta_i)$ is fixed for $1 \leq i \leq 4$. Note that this map is quasiconformal with constant Beltrami coefficient in each of the eight sectors of each hemisphere. Applying these maps to each hemisphere, we obtain a locally quasiconformal map $\phi: \mathcal{R}_0 \to \mathcal{R}$. The action of $\phi$ on a single hemisphere is shown in Figure~\ref{fig:phi}.

\begin{figure}[h]

\centering

\begin{tikzpicture}[scale=0.5]
    \def\maxradius{3}
    \def\ncircles{4}
    \def\nradials{3}

    \def\thetaone{10}
    \def\thetatwo{80}
    \def\thetathree{180}
    \def\thetafour{275}

    \pgfmathsetmacro\mone{(\thetaone+\thetatwo)/2}
    \pgfmathsetmacro\mtwo{(\thetathree+\thetatwo)/2}
    \pgfmathsetmacro\mthree{(\thetathree+\thetafour)/2}
    \pgfmathsetmacro\mfour{(\thetafour+\thetaone+360)/2}
    
    \begin{scope}[shift={(-7,0)}]
        
        \foreach \r in {1,...,\ncircles} {
            \draw (0,0) circle (\r*\maxradius/\ncircles);
        }
        \foreach \i in {1,...,\nradials} {
        
            \pgfmathsetmacro\angle{\thetaone*\i/\nradials+\mone*(\nradials-\i)/\nradials}
            \draw[dotted] (0,0) -- (\angle:\maxradius);

            \pgfmathsetmacro\angle{\thetatwo*\i/\nradials+\mone*(\nradials-\i)/\nradials}
            \draw[dotted] (0,0) -- (\angle:\maxradius);

            \pgfmathsetmacro\angle{\thetatwo*\i/\nradials+\mtwo*(\nradials-\i)/\nradials}
            \draw[dotted] (0,0) -- (\angle:\maxradius);

            \pgfmathsetmacro\angle{\thetathree*\i/\nradials+\mtwo*(\nradials-\i)/\nradials}
            \draw[dotted] (0,0) -- (\angle:\maxradius);

            \pgfmathsetmacro\angle{\thetathree*\i/\nradials+\mthree*(\nradials-\i)/\nradials}
            \draw[dotted] (0,0) -- (\angle:\maxradius);

            \pgfmathsetmacro\angle{\thetafour*\i/\nradials+\mthree*(\nradials-\i)/\nradials}
            \draw[dotted] (0,0) -- (\angle:\maxradius);

            \pgfmathsetmacro\angle{\thetafour*\i/\nradials+\mfour*(\nradials-\i)/\nradials}
            \draw[dotted] (0,0) -- (\angle:\maxradius);

            \pgfmathsetmacro\angle{(\thetaone+360)*\i/\nradials+\mfour*(\nradials-\i)/\nradials}
            \draw[dotted] (0,0) -- (\angle:\maxradius);
            
        }
        \draw[thick] (0, 0) -- (\thetaone:\maxradius);
        \draw[fill=black] (\thetaone:\maxradius) circle (2pt);
        \node[right, font=\footnotesize] at (\thetaone:\maxradius) {$(1, \theta_1)$};

        \draw[thick] (0, 0) -- (\thetatwo:\maxradius);
        \draw[fill=black] (\thetatwo:\maxradius) circle (2pt);
        \node[above, font=\footnotesize] at (\thetatwo:\maxradius) {$(1, \theta_2)$};

        \draw[thick] (0, 0) -- (\thetathree:\maxradius);
        \draw[fill=black] (\thetathree:\maxradius) circle (2pt);
        \node[left, font=\footnotesize] at (\thetathree:\maxradius) {$(1, \theta_3)$};

        \draw[thick] (0, 0) -- (\thetafour:\maxradius);
        \draw[fill=black] (\thetafour:\maxradius) circle (2pt);
        \node[below, font=\footnotesize] at (\thetafour:\maxradius) {$(1, \theta_4)$};

        \draw[gray] (0, 0) -- (\mone:\maxradius);
        \draw[fill=black] (\mone:\maxradius) circle (1pt);
        \node[right, font=\footnotesize] at (\mone:\maxradius) {$(1, m_1)$};

        \draw[gray] (0, 0) -- (\mtwo:\maxradius);
        \draw[fill=black] (\mtwo:\maxradius) circle (1pt);
        \node[left, font=\footnotesize] at (\mtwo:\maxradius) {$(1, m_2)$};

        \draw[gray] (0, 0) -- (\mthree:\maxradius);
        \draw[fill=black] (\mthree:\maxradius) circle (1pt);
        \node[left, font=\footnotesize] at (\mthree:\maxradius) {$(1, m_3)$};

        \draw[gray] (0, 0) -- (\mfour:\maxradius);
        \draw[fill=black] (\mfour:\maxradius) circle (1pt);
        \node[right, font=\footnotesize] at (\mfour:\maxradius) {$(1, m_4)$};

    \end{scope}

    \draw[->, thick] (-1.5, 0) -- (1.5, 0) node[midway, above] {$\phi$};

    \pgfmathsetmacro\mone{60}
    \pgfmathsetmacro\mtwo{150}
    \pgfmathsetmacro\mthree{200}
    \pgfmathsetmacro\mfour{315}

    \begin{scope}[shift={(7,0)}]
        \foreach \r in {1,...,\ncircles} {
            \draw (0,0) circle (\r*\maxradius/\ncircles);
        }
        \foreach \i in {1,...,\nradials} {
        
            \pgfmathsetmacro\angle{\thetaone*\i/\nradials+\mone*(\nradials-\i)/\nradials}
            \draw[dotted] (0,0) -- (\angle:\maxradius);

            \pgfmathsetmacro\angle{\thetatwo*\i/\nradials+\mone*(\nradials-\i)/\nradials}
            \draw[dotted] (0,0) -- (\angle:\maxradius);

            \pgfmathsetmacro\angle{\thetatwo*\i/\nradials+\mtwo*(\nradials-\i)/\nradials}
            \draw[dotted] (0,0) -- (\angle:\maxradius);

            \pgfmathsetmacro\angle{\thetathree*\i/\nradials+\mtwo*(\nradials-\i)/\nradials}
            \draw[dotted] (0,0) -- (\angle:\maxradius);

            \pgfmathsetmacro\angle{\thetathree*\i/\nradials+\mthree*(\nradials-\i)/\nradials}
            \draw[dotted] (0,0) -- (\angle:\maxradius);

            \pgfmathsetmacro\angle{\thetafour*\i/\nradials+\mthree*(\nradials-\i)/\nradials}
            \draw[dotted] (0,0) -- (\angle:\maxradius);

            \pgfmathsetmacro\angle{\thetafour*\i/\nradials+\mfour*(\nradials-\i)/\nradials}
            \draw[dotted] (0,0) -- (\angle:\maxradius);

            \pgfmathsetmacro\angle{(\thetaone+360)*\i/\nradials+\mfour*(\nradials-\i)/\nradials}
            \draw[dotted] (0,0) -- (\angle:\maxradius);
            
        }
        \draw[thick] (0, 0) -- (\thetaone:\maxradius);
        \draw[fill=black] (\thetaone:\maxradius) circle (2pt);
        \node[right, font=\footnotesize] at (\thetaone:\maxradius) {$(1, \theta_1)$};

        \draw[thick] (0, 0) -- (\thetatwo:\maxradius);
        \draw[fill=black] (\thetatwo:\maxradius) circle (2pt);
        \node[above, font=\footnotesize] at (\thetatwo:\maxradius) {$(1, \theta_2)$};

        \draw[thick] (0, 0) -- (\thetathree:\maxradius);
        \draw[fill=black] (\thetathree:\maxradius) circle (2pt);
        \node[left, font=\footnotesize] at (\thetathree:\maxradius) {$(1, \theta_3)$};

        \draw[thick] (0, 0) -- (\thetafour:\maxradius);
        \draw[fill=black] (\thetafour:\maxradius) circle (2pt);
        \node[below, font=\footnotesize] at (\thetafour:\maxradius) {$(1, \theta_4)$};

        \draw[gray] (0, 0) -- (\mone:\maxradius);
        \draw[fill=black] (\mone:\maxradius) circle (1pt);
        \node[above right, font=\footnotesize] at (\mone:\maxradius) {$(1, \alpha_1)$};

        \draw[gray] (0, 0) -- (\mtwo:\maxradius);
        \draw[fill=black] (\mtwo:\maxradius) circle (1pt);
        \node[left, font=\footnotesize] at (\mtwo:\maxradius) {$(1, \alpha_2)$};

        \draw[gray] (0, 0) -- (\mthree:\maxradius);
        \draw[fill=black] (\mthree:\maxradius) circle (1pt);
        \node[below left, font=\footnotesize] at (\mthree:\maxradius) {$(1, \alpha_3)$};

        \draw[gray] (0, 0) -- (\mfour:\maxradius);
        \draw[fill=black] (\mfour:\maxradius) circle (1pt);
        \node[right, font=\footnotesize] at (\mfour:\maxradius) {$(1, \alpha_4)$};

    \end{scope}

\end{tikzpicture}

\caption{The action of $\phi$ on one hemisphere}
\label{fig:phi}

\end{figure}
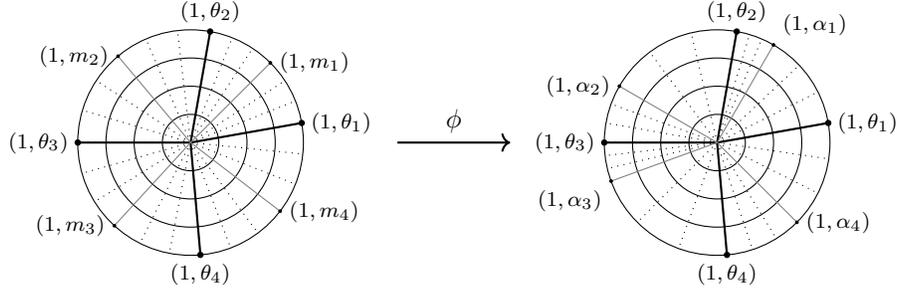

Let $\mathcal{P}$ be the partition each region of which is the set of eight triangular regions adjacent to a vertex, and let $\mu$ be the Beltrami coefficient of $p \circ \phi \circ \wp_0$. The measures $\nu_i$ introduced in Section~\ref{random differentials} are defined on $\mathbb{R}$, whereas the measures $\eta_i$ are defined on $C_i$. This causes no difficulty, since $C_i$ is homeomorphic to $\mathbb{R}$. Note that by the composition rules for Beltrami coefficients, the coefficient $\mu$ is no longer constant on each triangular region; however, $|\mu|$ is constant on each triangular region and depends continuously on $\alpha_i$. It follows that the function $k_t^{(i)}$ from Section~\ref{random differentials} can be chosen to be measurable in $t$; in fact, it can be chosen to be continuous.

Consider the random locally quasiconformal map $w^\mu$. If we set $f=p \circ \phi \circ \wp_0 \circ (w^\mu)^{-1}$, then the composition rules of Beltrami coefficients imply that $f$ is conformal. Furthermore $\mathcal{R}$ is the Riemann surface of $f^{-1}$ (see Figure~\ref{fig:f}).

\begin{figure}[h!]
\centering

\begin{tikzcd}
\mathbb{C} \arrow[r, "\wp_0"] \arrow[d, "w^\mu"] & \mathcal{R}_0 \arrow[r, "\phi"] & \mathcal{R} \arrow[d, "p"] \\
w^\mu(\mathbb{C}) \arrow[rr, dashed, "f"] & & \overline{\mathbb{C}}
\end{tikzcd}

\caption{$\mathcal{R}$ is the Riemann surface of $f^{-1}$.}
\label{fig:f}

\end{figure}

The partition $\mathcal{P}$ is periodic and has bounded geometry, and $\mu$ is periodic on $\mathcal{P}$. Thus Lemma~\ref{periodic=>bounded} implies that $\mu$ is probabilistically bounded, so  Theorems~\ref{almost sure surjectivity} and~\ref{approximate linearity} are applicable. We now obtain a quick proof of Theorem~\ref{parabolicity}.

\begin{proof}[Proof of Theorem~\ref{parabolicity}]
    Note that $\phi \circ \wp_0 \circ (w^\mu)^{-1}$ is a conformal homeomorphism from $w^\mu (\mathbb{C})$ to $\mathcal{R}$. By Theorem~\ref{almost sure surjectivity}, we have that $w^\mu$ is almost surely surjective onto $\mathbb{C}$, so $\phi \circ \wp_0 \circ (w^\mu)^{-1}$ is almost surely a biholomorphism from $\mathbb{C}$ to $\mathcal{R}$.
\end{proof}

We now want to estimate the order of  $f$. Since Theorem~\ref{approximate linearity} is applicable, we are ready for the proof of Theorem~\ref{order of growth estimate}.

\begin{proof}[Proof of Theorem~\ref{order of growth estimate}]

Let $\sigma$ be the spherical metric on $\overline{\mathbb{C}}$ scaled so that $\overline{\mathbb{C}}$ has area 1. We define $A(t)$ to be the area of $B(0, t)$ in the pullback metric $f^*\sigma$. According to a well known formula of Ahlfors and Shimizu (Equation~12 in~\cite{intro to nevanlinna}), $$T_f(r) = \int_0^r \frac{A(t)}{t} dr + O(1).$$ We thus need to estimate the growth of $A(t)$.

Let $D(0, R)$ denote the ellipse such that $A_\mu(D(0, R))=B(0, R)$, where $A_\mu$ is the linear map whose existence is guaranteed by Theorem~\ref{approximate linearity}. Note that $D(0, R)$ is $D(0, 1)$ scaled by a factor of $R$. By the remark at the end of the proof of Theorem~\ref{approximate linearity}, the conclusion of Theorem~\ref{approximate linearity} holds with $B(0, R)$ replaced by $D(0, R)$. Fix $\epsilon>0$. We may choose a sequence $R_n$ such that the conclusion of Theorem~\ref{approximate linearity} holds for $R_n$ with probability at least $1-\frac{\epsilon}{2^n}$ for all positive integers $n$. By the Borel-Cantelli lemma, with probability 1 the conclusion of Theorem~\ref{approximate linearity} holds for infinitely many $n$. The following deductions hold for these values of $n$.

For infinitely many values of $R_n$, for $t \in \left[\frac{R_n}{2}, R_n\right]$,
$$I_{1-2\epsilon} (A_\mu(D(0, t))) \subset w^\mu(D(0, t)) \subset I_{1+2\epsilon} (A_\mu(D(0, t)))$$
with probability 1. Using the definition of $D(0, t)$, we may rewrite this as
$$B(0, t(1-2\epsilon)) \subset w^\mu(D(0, t)) \subset B(0, t(1+2\epsilon)),$$
or alternatively as 
$$(w^\mu)^{-1}(B(0, t(1-2\epsilon))) \subset D(0, t) \subset (w^\mu)^{-1}(B(0, t(1+2\epsilon))).$$
It thus follows that for $t \in \left[\frac{R_n}{2}(1+2\epsilon), R_n(1-2\epsilon)\right]$,
\begin{equation}\label{eq: inclusion}
    D\left(0, \frac{t}{1+2\epsilon}\right) \subset (w^\mu)^{-1}(B(0, t)) \subset D\left (0, \frac{t}{1-2\epsilon}\right)
\end{equation}
For simplicity, from now on we take $\epsilon<\frac{1}{20}$ small enough so that the above holds for $t \in [0.6R_n, 0.9R_n]$.

We now turn to bounds on $A(t)$. There exist constants $m$ and $M$ such that for large enough $t$,  $D(0, t)$ completely contains at least $mt^2$ cells of the rectangular lattice in its interior and is contained in the union of $Mt^2$ cells. Recall that we are interested in the function $f=p \circ \phi \circ \wp_0 \circ (w^\mu)^{-1}$. For $t \in [0.6 R_n, 0.9 R_n]$, Expression~\ref{eq: inclusion} implies
$$(w^\mu)^{-1}(B(0, t)) \subset D \left (0, \frac{t}{1-2\epsilon} \right ),$$
from which it follows that the image of $B(0, t)$ by $(w^\mu)^{-1}$ is contained in the union of at most $M(\frac{t}{1-2\epsilon})^2$ cells. Thus the image of $B(0, t)$ by $\wp_0 \circ (w^\mu)^{-1}$ is covered by at most $M(\frac{t}{1-2\epsilon})^2$ hemispheres. Since $\phi$ maps hemispheres to hemispheres and each hemisphere has area $\frac{1}{2}$, it follows that $$A(t) \leq \frac{Mt^2}{2(1-2\epsilon)^2}.$$

We construct an analogous lower bound on $A(t)$. We have $$(w^\mu)^{-1}(B(0, t)) \supset D\left(0, \frac{t}{1+2\epsilon}\right)$$ for $t \in [0.6 R_n, 0.9 R_n]$ for infinitely many $R_n$ with probability 1. It follows that the image of $B(0, t)$ by $(w^\mu)^{-1}$ contains at least $m(\frac{t}{1+2\epsilon})^2$ cells. Thus the image of $B(0, t)$ by $\wp_0 \circ (w^\mu)^{-1}$ and $\phi \circ \wp_0 \circ (w^\mu)^{-1}$ contains at least $m(\frac{t}{1+2\epsilon})^2$ hemispheres. Hence
$$A(t) \geq \frac{mt^2}{2(1+2\epsilon)^2}.$$

In summary, we have that for some constants $c_1$ and $c_2$,
$$c_1t^2 \leq A(t) \leq c_2t^2$$
for $t \in [0.6 R_n, 0.9 R_n]$ for infinitely many $n$ with probability $1$. Integrating, it follows that
$$T_f(R_n)+O(1)=\int_0^{R_n} \frac{A(t)}{t} dt \geq \int_{0.6 R_n}^{0.9 R_n} \frac{A(t)}{t} dt \geq \int_{0.6R_n}^{0.9R_n} c_1t  dt = 0.225 c_1R_n^2.$$
Denoting the order of growth of $f$ by $\lambda$, it follows that
$$\lambda = \limsup_{r \to \infty} \frac{\log T_f(r)}{\log r} \geq \limsup_{n \to \infty} \frac{\log T_f(R_n)}{\log R_n} \geq \limsup_{n \to \infty} \frac{2 \log R_n + O(1)}{\log R_n} = 2.$$
with probability 1.

Similarly, we have that
$$T_f(0.9 R_n)+O(1)=\int_0^{0.9R_n} \frac{A(t)}{t} dt \leq \int_0^{0.9R_n} c_2t dt = 0.405c_2R_n^2,$$
where we used the fact that $A(t)$ is an increasing function. Denoting the lower order by $\underline{\lambda}$, it follows that
$$\underline{\lambda} = \liminf_{r \to \infty} \frac{\log T_f(r)}{\log r}  \leq \liminf_{n \to \infty} \frac{\log T_f(0.9R_n)}{\log 0.9R_n} \leq \liminf_{n \to \infty} \frac{2 \log R_n + O(1)}{\log R_n+O(1)} = 2$$
with probability 1. The proof of Theorem~\ref{order of growth estimate} is complete.

\end{proof}

\section{Directions for Further Research}\label{further discussion}

One of the main limitations of Theorems~\ref{parabolicity} and~\ref{order of growth estimate} are that they only establish results about a subset $\mathcal{V}$ of $\mathcal{S}$. The main challenge in constructing a measure on all of $\mathcal{S}$ is that $\mathcal{S}$ is difficult to parametrize: simply specifying each marked boundary point does not necessarily define a surface in $\mathcal{S}$, since the marked boundary points associated to the vertices of a cell must be in cyclic order.

One idea for parametrizing all of $\mathcal{S}$ is to specify the vertices of each hemisphere up to biholomorphism. More specifically, let $a, b, c, d$ be the marked boundary vertices of a hemisphere in cyclic order, so that there exists a unique automorphism of the hemisphere mapping $b, c, d$ to $1, 0, \infty$, given by
$$z \mapsto \frac{(z-c)(b-d)}{(z-d)(b-c)}.$$
The cross ratio of $a, b, c, d$ is defined as the image of $a$ under this map; it is thus given by the formula
$$(a, b; c, d) = \frac{(a-c)(b-d)}{(a-d)(b-c)}.$$
Since $a, b, c, d$ are in cyclic order, it follows that $(a, b; c, d)$ belongs to the interval $(1, \infty)$. If $H(a, b, c, d)$ denotes the hemisphere with the marked boundary points $a, b, c, d$, it follows that there exists a biholomorphism between $H(a, b, c, d)$ and $H(a', b', c', d')$ mapping corresponding marked boundary points to each other if and only if $(a, b; c, d) = (a', b'; c', d')$. Thus each hemisphere can be specified up to biholomorphism by its cross ratio, which is an arbitrary number from $(1, \infty)$.

The natural question is whether specifying the cross ratio of each cell of a surface $\mathcal{R} \in \mathcal{S}$ determines $\mathcal{R}$ up to biholomorphism. This is phrased more specifically in the following conjecture.

\begin{conjecture}\label{cross ratio conjecture}
    Suppose $\mathcal{R}$ and $\mathcal{R}'$ are two surfaces in $\mathcal{S}$ such that the cell of $\mathcal{R}$ at position $(i, j)$ has cross ratio $k_{i, j}$ and the cell of $\mathcal{R}'$ at position $i, j$ has cross ratio $k_{i, j}'$. Suppose furthermore that $k_{i, j}=k_{i, j}'$ for all $i, j$. Then $\mathcal{R}$ and $\mathcal{R}'$ are of the same conformal type, meaning that $\mathcal{R}$ and $\mathcal{R}'$ are equivalent up to biholomorphism.
\end{conjecture}

Conjecture~\ref{cross ratio conjecture} is non-trivial to prove. Indeed, even though there exist biholomorphisms $\phi_{i, j}$ from the cell of $\mathcal{R}$ at position $(i, j)$ to the cell of $\mathcal{R}'$ at position $(i, j)$, the $\phi_{i, j}$ do not glue into a biholomorphism $\phi: \mathcal{R} \to \mathcal{R}'$ because the $\phi_{i, j}$ do not necessarily agree at the boundary of neighboring cells. A proof of Conjecture~\ref{cross ratio conjecture} would allow one to associate a surface in $\mathcal{S}$ with well defined conformal type to every element of the countably infinite product space $(1, \infty)^\omega$. From here it is easier to define a probability distribution on the set of surfaces with a square grid net, for example by choosing each cross ratio independently and at random with the same probability distribution on each cell. This would open up the possibility of proving that almost every surface in $\mathcal{S}$ is parabolic, without having to restrict to a subspace $\mathcal{V}$.

It is also of interest to know whether the proof of Theorem~\ref{order of growth estimate} can be strengthened to a proof of the following conjecture.

\begin{conjecture}\label{order of growth conjecture}
    The meromorphic function $f$ defined by a surface in $(\mathcal{V}, \eta_\mathcal{V})$ almost always  has order 2.
\end{conjecture}

We have not been able to obtain upper bounds on $\limsup \frac{\log T_f(r)}{\log r}$ or lower bounds on $\liminf \frac{\log T_f(r)}{\log r}$ that hold almost surely. Proving Conjectures~\ref{cross ratio conjecture} and~\ref{order of growth conjecture} is an interesting further direction of research.

\section{Acknowledgements}

I would like to thank my mentor, Professor Sergiy Merenkov, for proposing the project and providing his valuable guidance and feedback. This paper would not have been possible without him. I would also like to express my heartfelt gratitude for the MIT PRIMES program for the wonderful opportunity to conduct research. This was a one-of-a-kind learning experience from which I benefited enormously.

\end{document}